\documentclass[12pt,reqno]{amsart}
\usepackage{a4}
\usepackage{amsmath}
\usepackage{amssymb}
\usepackage{amscd}                  
\usepackage{verbatim}
\usepackage[german, french, english]{babel}
\usepackage[utf8]{inputenc}
\usepackage{graphicx}                     
\newcounter{theorem}
\newtheorem{cor}[theorem]{Corollary}
\newtheorem{lem}[theorem]{Lemma}
\newtheorem{thm}{Theorem}
\newtheorem{pro}[theorem]{Proposition}
\newtheorem{dfn}[theorem]{Definition}
\newtheorem{exa}[theorem]{Example}

\newtheorem{rem}[theorem]{Remark}

\newcounter{intro}

\newcommand{\cref}[1]{Corollary~\ref{#1}}

\newcommand{\eref}[1]{(\ref{#1})}

\newcommand{\lref}[1]{Lemma~\ref{#1}}
\newcommand{\pref}[1]{Proposition~\ref{#1}}

\newcommand{\ccup}{\cup\hspace{-6pt}{^\uparrow}}

\newcommand{\tir}{\discretionary{.}{}{---\kern.7em}}

\newcommand\I{\mathbb{I}}
\newcommand\C{\mathbb C}

\newcommand\N{\mathbb{N}}
\newcommand\R{\mathbb R}
\newcommand\Z{\mathbb{Z}}

\newcommand\mcV{\mathcal V}  \newcommand\mcE{\mathcal E}

  \newcommand\mcH{\mathcal H}
  
 \newcommand\mcZ{\mathcal Z}

\newcommand{\di}{\mathrm d} %
\renewcommand{\phi}{\varphi}   
\newcommand{\wt}{\widetilde}           
\newcommand{\wbar}{\overline}           

\newcommand{\Dom}{\operatorname{Dom}}

\renewcommand{\Im}{\operatorname{Im}}

\newcommand{\Ker}{\operatorname{Ker}}

\renewcommand{\Re}{\operatorname{Re}}

\newcommand{\supp}{\operatorname{supp}}

\begin{document}

\title[Gau\ss-Bonnet operator on graphs]
{The Gau\ss-Bonnet operator of an infinite graph}

\author{Colette Ann\'e}
\address{Laboratoire de Math\'ematiques Jean Leray, Universit\'e de Nantes,
 CNRS, Facult\'e des Sciences, BP 92208, 44322 Nantes, France}
\email{colette.anne@univ-nantes.fr}

\author{Nabila Torki-Hamza}
\address{ISIG-K, Universit\'e de Kairouan, 3100-Kairouan; Tunisie}
\email{nabila.torki-hamza@fsb.rnu.tn; natorki@gmail.com }

\date{\today \;
\emph{File: }\texttt{AT-cor.tex}
\\{$2010$ {\it Mathematics Subject Classification.} 
 39A12, 05C63, 47B25, 05C12, 05C50. } 
}

\begin{abstract}
\tir We propose a general condition, to ensure  essential self-adjointness
for the Gau\ss-Bonnet operator $D=d+\delta$, based on a notion
of completeness as Chernoff.  
This gives essential self-adjointness 
of the Laplace operator both for functions and 1-forms on infinite graphs.
This is used to extend Flanders result concerning solutions of Kirchhoff's 
laws. 

 \hspace{-0.5cm}
 {\sc R\'esum\'e.} Nous proposons une condition g\'en\'erale qui assure le 
caract\`ere essentiellement auto-adjoint de l'op\'erateur de Gauss-Bonnet $D=d+\delta$, bas\'ee 
sur une notion de compl\'etude comme Chernoff. 
Comme cons\'equence, 
l'op\'erateur de Laplace agissant sur les fonctions et les 1-formes de graphes 
infinis est essentiellement auto-adjoint. Nous utilisons ce cadre pour \'etendre le r\'esultat de Flanders \`a propos des solutions des lois de Kirchhoff.

\end{abstract}

\maketitle



\section{Introduction}
 Operators on infinite graphs are of large interest
and a lot of recent works deals with this subject. One approach can be to 
study how techniques of spectral geometry can be extended on graphs regarded 
as one-dimensional simplicial complexes. We refer to Dodziuk \cite{D84,DK} for
 general presentation of this approach and to \cite{CdV,CTT} for the 
geometric point of view, and also \cite{CdV0} for the relation between
Kirchhoff's laws and Hodge theory. 

We consider here only connected locally finite infinite graphs and we study 
Kirchhoff's laws.
Flanders has first studied this question on infinite graphs seen as infinite 
electric networks, see \cite{F}. 
Several authors have  clarified and extended Flanders work on electric 
networks, see for instance Thomassen \cite{T90}, Soardi \cite{S},  
Doyle \& Snell \cite{D}, Zemanian 
\cite{Z}, Georgakopoulos \cite{G}, Carmesin \cite{Cm} and also the book of
Jorgensen \& Pearse \cite{JP} for a general approach.\\
Flanders main result is that there exists a unique current flow in an infinite 
network with a finite number of sources which is the limit of flows with finite support.

In our paper, this question is approached by the study of a Dirac type 
operator: {\it the Gau\ss-Bonnet operator} $D=d+\delta$, introduced on an 
infinite graph considered as a one-dimentional simplicial complex. 
Indeed, this operator is a generalisation of the Dirac operator studied on 
$\Z$ by Golenia \& Haugomat in \cite{GH}.
We give  a general condition on the graph by defining the notion of  
\emph{$\chi-$completeness}, see Section \ref{sa}. 
One of the main results  is essential self-adjointeness of the 
Gau\ss-Bonnet operator, when the graph is $\chi-$complete 
(or {\it complete homogeneous}). This condition covers the situations of  \cite{M}, 
and \cite{T} (or \cite{T12}), it is satisfied by locally finite graphs
which are complete for some intrinsic pseudo metric, as defined in \cite{HKMW}
(although the results of \cite{HKMW} are valid in a more general context of graphs not 
necessarily locally finite), 
and it is a discrete version of a result of Chernoff, see  \cite{Ch}, in the case of 
manifolds. 
One of the applications in his paper concludes that, on a complete manifold, 
every power of the Dirac operator $d+\delta$ is essentially self-adjoint. 
In particular, for every power of  the Laplace-Beltrami operator, essential 
self-adjointness is true.

In Section \ref{anghel}, we define the property of {\it positivity at infinity} 
for Dirac type operators. And by adding this assumption on our Gau\ss-Bonnet 
operator, we prove that its range  is closed and consequently the Hodge 
property holds, in a similar result as Anghel's for compact Riemannian 
manifold, see \cite{A}. 
This situation  permits us to enlarge the conditions on the current source
and the voltage source in the Flanders problem. In Section \ref{examples}, 
we give new examples of infinite graphs where it applies.

\section{Preliminaries} 
\subsection{Definitions on Graphs}(cf. \cite{LP})
 A graph $K$ is a simplicial complex of dimension one. 
We denote by $\mcV$ the set of vertices and $\mcE$ the set of 
{\it oriented edges}, considered as a 
subset of $\mcV\times\mcV$.
We assume that $\mcE$ is symmetric without loops:
$$v\in\mcV\Rightarrow (v,v)\notin \mcE,\quad (v_1,v_2)\in \mcE\Rightarrow
(v_2,v_1)\in \mcE.$$
Choosing an orientation of the graph consists of defining a partition of 
$\mcE: $ 
\begin{align*}
&\mcE^+\sqcup\mcE^-=~\mcE \\
&(v_1,v_2)\in \mcE^+\Longleftrightarrow(v_2,v_1)\in \mcE^-. 
\end{align*} 
For $e=(v_1,v_2)\in\mcE,$ let's set 
$$e^+=v_2,\, e^-=v_1,\, -e=(v_2,v_1).
$$
$e^+$ and $ e^-$ are called boundary points of the edge $e$.
\subsubsection{} 
A {\it path} between two vertices $x,y$ in $\mcV$ is a 
finite set of edges
$e_1,\dots,e_n, n\geq 1$ such that 
$$
e_1^-=x,\, e_n^+=y~ \hbox{ and, if }~ n\geq 2, ~\forall j,~~1\leq j\leq (n-1) 
\Rightarrow e_j^+=e_{j+1}^-.
$$
Notice that each path has a beginning and an end, and that an edge is a path.\\ 
Let us denote $\Gamma_{xy}$ the set of the
paths from the vertex $x$ to the vertex $y.$

\subsubsection{} 
The graph is {\it connected} if two vertices are always 
related by a path, {\it ie.} if $\Gamma_{xy}$ is non empty for all $x,y$ in $\mcV .$
\subsubsection{} 
The graph is {\it locally finite} if each vertex belongs 
to a finite number of edges. The {\it degree} or {\it valence} of a vertex $x\in\mcV$
is the cardinal of the set $\{e\in\mcE;\,e^+=x\}.$ 
\subsubsection{}\label{subgraph}
A {\it subgraph} of a graph $K$ is a graph $K_0=(\mcV_0,\mcE_0)$ such that
$\mcV_0\subset\mcV$ and $\mcE_0\subset \mcE.$

\begin{rem} 
All the graphs we shall consider on the sequel will be connected, 
locally finite, so with countably many vertices.
\end{rem}
\subsection{Functions and forms}
The $0-$cochains are just scalar functions on $\mcV$, we denote the set by $C^0(K).$\\
The $1-$cochains or forms are odd scalar functions on $\mcE$ we denote the set by $C^1(K).$\\
Thus we have
\begin{align*}
C^0(K)= ~&\C^\mcV,\\
 C^1(K)= ~&\{\phi:\mcE\to\C,\phi(-e)=-\phi(e)\}.
\end{align*}
The sets of cochains with finite support are denoted by  $C_0^0(K),~C_0^1(K)$.
\\
To obtain Hilbert spaces we need weights, let us give
$$c:\mcV\to \R_+^{\ast},~$$ 
and  $$r:\mcE\to \R_+^{\ast} ~ \hbox{even}
$$
so $r(-e)=r(e)$. \\
They define scalar products: 

\begin{align*}
\forall ~f,g\in C_0^0(K)\, ;~ &\langle f,g\rangle =\sum_{v\in\mcV}c(v)f(v)\bar g(v)
\end{align*}
\begin{align}\label{e-scalar}
\forall~ \phi,\psi\in C_0^1(K)\, ; ~ &\langle \phi,\psi\rangle=\frac 12\sum_{e\in\mcE}r(e)
\phi(e)\bar\psi(e) 
\end{align}
\begin{rem}
As the products  $r(e)\phi(e)\bar\psi(e), ~ e\in \mcE$  in \eref{e-scalar} are even, the term $\dfrac 12$ allows 
to recover the usual definition.
\end{rem}
\begin{rem}
In the context of electric networks, our weight on edges would play the role
of the {\it conductance}, the intensity would be on $e\in\mcE:I(e)=r(e)\phi(e)$
and the energy $\|\phi\|^2=\frac 12\sum_{e\in\mcE}\frac{1}{r(e)}I(e)^2.$ 
So, indeed, $\frac{1}{r(e)}$ is the resistance of the edge $e$!
\end{rem}

Let us finally define the Hilbert spaces $L_2(\mcV)$ and $ L_2(\mcE)$ as the sets of cochains with finite norm, 
we have
\begin{align*}
L_2(\mcV)=\overline{C_0^0(K)},\\
 L_2(\mcE)=\overline{C_0^1(K)}.
\end{align*}

and put
\begin{equation}\label{H}
\mcH=L_2(\mcV)\oplus L_2(\mcE), \forall F=(f,\phi)\in\mcH,\, \|F\|^2=\|f\|^2+\|\phi\|^2.
\end{equation}
{\it Comment.} $L_2(\mcV)$ and $ L_2(\mcE)$ can be considered as subspaces of $\mcH,$ this justifies that
all the $L_2$-norms have the same notation. 
\subsection{Operators}
\subsubsection{The difference operator} It is the operator
$$\di:C_0^0(K)\to C_0^1(K),$$ given by
\begin{equation}\label{di} 
 \di(f)(e)=f(e^+)-f(e^-),
\end{equation}
for $f\in C_0^0(K), e\in \mcE. $
\subsubsection{The coboundary operator}
 It is $\delta $ the formal adjoint of $\di .$ \\
Thus it satisfies 
\begin{equation} \label{delta}
\langle \di f,\phi\rangle=\langle f,\delta\phi\rangle
\end{equation} 
for all $f\in C_0^0(K)$ and $\phi\in C_0^1(K).$
\begin{lem}
The coboundary operator $\delta:C_0^1(K)\to C_0^0(K),\;$ acts as
\begin{equation}\label{del}  
\delta(\phi)(x)=\frac{1}{c(x)}
\sum_{e,e^+=x}r(e)\phi(e).
\end{equation}
\end{lem}
\begin{proof} \tir
Using the equation \eref{delta}, we have
$$\frac 12\sum_{e \in\mcE } r(e)\left(f(e^+)-f(e^-)\right)\bar\phi(e)=\frac 12\sum_{x\in\mcV} f(x)
\wbar{\Big(\sum_{e^+=x}r(e)\phi(e)-\sum_{e^-=x}r(e)\phi(e)\Big)}
$$
But $r\phi$ is odd and $\mcE$ symmetric, so 
$$\sum_{e^-=x}r(e)\phi(e)=-\sum_{e^+=x}r(e)\phi(e).
$$
We remark that the sum entering in the formula \eref{del} of $\delta$ is finite
due to the hypothesis that the graph is locally finite.
\end{proof}
\begin{rem}
The operator $\di$ is defined by \eref{di} on $C^0(K),$ but to define
$\delta$ on $C^1(K),$ we need an hypothesis on $K$: we suppose that
the graph is locally finite. This hypothesis could be weakened by assuming that 
the edge weights $r(e),~~ e \in \mcE$ are summable around each vertex as considered in \cite{KLe}.
\end{rem}
With these two operators we can define the following two operators.
\subsubsection{The Gau\ss-Bonnet operator}
 It is the endomorphism 
$$D=\di+\delta:C_0^0(K)\oplus C_0^1(K)\to C_0^0(K)\oplus C_0^1(K)
$$
given by
$$ D(f, \varphi)= (\delta \varphi , \di f)
$$
for all $f \in C_0^0(K)$ and $\varphi \in  C_0^1(K).$\\
\begin{rem}On a locally finite graph, the operator $D$ extends to $C^0(K)\oplus C^1(K)$ and we still denote
it $D$ if there is no confusion.
\end{rem}
The domain $C_0^0(K)\oplus C_0^1(K)$ of $D$ is dense in the Hilbert space $\mcH$ (defined in \eref{H}). 
This operator is symmetric and of Dirac type, {\it i.e.} $D^2$ is of Laplace type.
\subsubsection{Laplacian} 
By definition, it is $$\Delta=D^2 :C_0^0(K)\oplus C_0^1(K)\circlearrowleft.$$
This operator preserves the direct sum $C_0^0(K)\oplus C_0^1(K),$
so we can write 
$$\Delta=\Delta_0\oplus \Delta_1.
$$ 

\subsection{Metrics}\label{met}
By analogy to Riemannian geometry, we call {\it metric} an even function 
$$a:\mcE\to\R_+^{\ast}.
$$ 
It defines a distance on the graph $K$ in the following way. \\
One first defines the {\it length of a path}:
for $\gamma=(e_1,\dots,e_n)$ 
$$l_a(\gamma)=\sum_{j=1}^n \sqrt{a(e_j)}.
$$
Then the {\it metric distance} between two vertices $x,y$ is given by
$$d_a(x,y)=\inf_{\gamma\in\Gamma_{xy}}l_a(\gamma).
$$

\section{Closability and Self-adjointness}
\subsection{Closability}
\begin{lem} 
If the graph $K$ is connected and locally finite the operators 
$\di$ and $\delta$ are closable.
\end{lem}
\begin{proof}\tir 
Let us suppose that there exists a sequence
$(f_n)_{n\in\N}$ in $C_0^0(K)$ such that $\|f_n\|\to 0$ and $(\di(f_n))_n$ converges.
Let us denote by $\phi$ this limit.\\
We have to show that $\phi=0.$ 
If $$\|f_n\|+\|\di(f_n)-\phi\|\to 0,
$$ 
then for each vertex $v$, $f_n(v)$ converges to $0$ and 
for each edge $e$, $\di(f_n)(e)$ 
converges to $\phi(e).$ But by the first statement and the expression of $\di,$
for each edge $e,\,\di(f_n)(e)$ converges to $0$.

The same can be done for $\delta$: convergence in norm to $0$ of a sequence 
$(\phi_n)_n$ implies pointwise convergence to $0$ which implies pointwise convergence 
of $\delta(\phi_n)$ to $0$, because of local finiteness of the graph ; if
$\delta(\phi_n)$ converges in norm, it must be to $0$.
\end{proof}
\vspace{0.5cm}
Thus, we can consider different extensions of these operators in the 
framework of Hilbert spaces (see \cite{RS}).\\

The smallest extension is the closure, denoted $\bar d=\di_{min}$ (resp. $\bar\delta=\delta_{min}$
and $\bar D=D_{min}$) has the domain
\begin{multline}
\Dom(\di_{min})=\{f\in L_2(\mcV); \,\exists (f_n)_{n\in\N},\, f_n\in C_0^0(K),\,
L_2\!-\lim_{n\to\infty}f_n=f,\\
L_2\!-\lim_{n\to\infty}\di(f_n) \hbox{ exists }\}
\end{multline}
for such an $f$, one puts 
$$\di_{min}(f)=\lim_{n\to\infty}\di(f_n).
$$

The largest is $\di_{max}=\delta^\ast,$ the adjoint operator of $\delta_{min},$ (resp.  
 $\delta_{max}=\di^\ast,$ the adjoint operator of $\di_{min}.)$
 

\subsection{A sufficient condition for self-adjointness of $D$}\label{sa}
\subsubsection{Geometric hypothesis for the graph $K$} 
\begin{dfn}
The graph $K$ is {\em $\chi-$complete}
if there exists a increasing sequence of finite sets 
$(B_n)_{n\in \N}$ such that
$\mcV=\ccup\, B_n$ and there exist related functions $\chi_n$ satisfying the
following three conditions:
\begin{enumerate}
\item[  (i)] $\chi_n\in C_0^0(K),\, 0\leq\chi_n\leq 1$
\item[ (ii)] $v\in B_n ~\Rightarrow ~\chi_n(v)=1$
\item[(iii)] $\displaystyle\exists C>0, \forall n\in \N,\, x\in\mcV,\,
\frac{1}{c(x)}\sum_{e,e^\pm=x}r(e)\di\chi_n(e)^2\leq C.$
\end{enumerate}
\end{dfn}

For this type of graphs one has
\begin{align}
\label{E0}\forall p\in \N,\,\exists n_p,\, n\geq n_p &\Rightarrow ~
\forall e\in\mcE,\hbox{ such that } e^+\hbox{ or }e^-\in B_p, \;\di\chi_n(e)=0
\\
\label{E}\mcE=\ccup\, \mcE_n &\hbox{ if } 
\mcE_n=\{e\in\mcE, e^+\in B_n\hbox{ or } e^-\in B_n \} 
\\
\forall f\in L_2(\mcV),\,&\lim_{n\to\infty}<\chi_n f,f>=\|f\|^2
\\
\label{E2}\forall \phi\in L_2(\mcE),\,& \|\phi\|^2=
\lim_{n\to\infty}\frac 12\sum_{e\in\mcE}r(e)\chi_n(e^+)|\phi(e)|^2
\\
\hbox {and} &\lim_{n\to\infty}\sum_{e\in\supp(\di\chi_n)}r(e)|\phi(e)|^2=0.
\end{align}

\begin{exa}\rm Let us consider an infinite tree with increasing valence:
\begin{figure}[ht]
\includegraphics*[height=6cm,width=12cm]{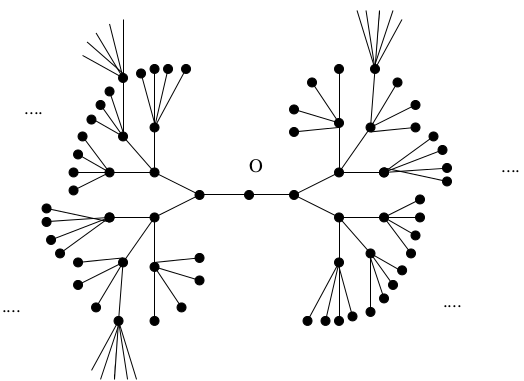}
\end{figure}

Taking constant weights on vertices and edges, this graph is $\chi-$complete.
Indeed, one can define {\it generations} of vertices on such a graph:
the considered origin vertex $O$ is of generation $0$ and valence $2$, it is
related to two vertices which are of generation $1$ and valence $3$, and more 
generally
there are $2n!$ vertices of generation $n$ and valence $(n+2).$

One defines $B_n,n\in \N$,  as the set of vertices of generation
less than $n^2$ and $\chi_n$ constant on each generation of vertices:
$$ \hbox{$x$ of generation $p$}\Rightarrow 
\chi_n(x)=\Big(\frac{(n+1)^2-p}{2n+1}\wedge 1\Big)\vee 0.
$$
So, $p\leq n^2\Rightarrow \chi_n(x)=1$ and 
$p\geq (n+1)^2\Rightarrow \chi_n(x)=0$ while $|d\chi_n(e)|\leq 1/(2n+1)$
is in fact supported on edges between generations larger than $n^2$
and less than $(n+1)^2.$ To verify the condition (iii), one has
to calculate for these generations, $(p+2)/(2n+1)^2\leq ((n+1)^2+2)/(2n+1)^2$ 
which is bounded independently on $n$.
\end{exa}

\begin{rem}\label{rema} 
The condition of $\chi-$completeness covers many situations that 
have been already studied. Particularily it is satisfied by locally finite graphs
which are complete for some intrinsic pseudo metric, as defined in \cite{HKMW}.
\end{rem}

\begin{lem}\label{intrinsic} If the graph admits an {\rm intrinsic path 
metric} $d$ such that $(\mcV,d)$ is complete, then the graph is 
$\chi-$complete. 
\end{lem}
\begin{proof}\tir The hypothesis means that our infinite, connected, locally finite, 
weighted graph admits a metric $a$ as defined in Section \ref{met} such that 
$$\forall x\in\mcV,\, \frac{1}{c(x)}\sum_{e\in\mcE,e^+=x}r(e)a(e)\leq 1
$$
(the relation between our notations and those of \cite{HKMW} is:
$\sigma^2=a$). We suppose also that the
metric distance $d_a$ defines  $(\mcV,d_a)$ as a complete metric space.

We then define the functions $\chi_n$ as follows. Fix $O$ a vertex in $\mcV$
and put
\begin{equation}\label{chin}
\forall n\in \N, \, B_n=\{x\in \mcV;\,d_a(O,x)\leq n\},\, 
\chi_n(x)=\sup\{(1-d_a(x,B_n)),0\}
\end{equation}
As pointed in \cite{HKMW} completeness of  $(\mcV,d_a)$ gives that the $B_n$
are finite. We verify that 
\begin{enumerate}
\item[  (i)] The support of $\chi_n$ is finite: it is included in 
$\{x;~d_a(x,B_n)\leq 1 \}\subset B_{n+1}.$
\item[ (ii)] $x\in B_n\Rightarrow d_a(x,B_n)=0\Rightarrow \chi_n(x)=1$

\item[(iii)] finally, by the triangle inequality, $d\chi_n(e)^2\leq a(e)$; then
the condition of {\it intrinsic metric} gives:
$$\forall x\in\mcV,\, \frac{1}{c(x)}\sum_{e\in\mcE,e^+=x}r(e)d\chi_n(e)^2\leq 1.
$$
\end{enumerate}

\end{proof}

\begin{rem}
If we consider the metric already introduced in \cite{CTT}  
(but to study non complete situations)
\begin{equation}\label{metric}
a(e)=\frac{\min(c(e^+),c(e^-))}{r(e)}
\end{equation} 
and with bounded valence:
$$\exists A>0,\, \forall v\in\mcV, \quad\sharp\{e\in\mcE,\, e^+=v\}\leq A.
$$
then, if the graph is complete for this metric, $\chi$-completeness is also 
satisfied. Indeed, $\dfrac{a}{A}$ is an intrinsic metric, because:
$$\forall x\in\mcV,\, \frac{1}{c(x)}\sum_{e\in\mcE,e^+=x}r(e)a(e)\leq A.
$$

It is  also the case in the situation of \cite{M} 
where the hypothesis taken give that 
$$\sup_{x\in B_n}\frac{1}{c(x)}\sum_{e,e^\pm=x}r(e)\di\chi_n(e)^2=o(1)
$$ 
for some
$\chi_n$ satisfying $\di\chi_n(e)^2=O(n^{-2}).$
\end{rem}
\begin{thm}\label{esa}
Let K be a connected, locally finite graph. If K is $\chi$-complete, 
then the operator $D$ is essentially self-adjoint.
\end{thm}
\begin{proof}\tir First note that
\subsubsection*{(a) If $\di_{min}=\di_{max}$ and $\delta_{min}=\delta_{max}$ then $D$ is 
essentially self-adjoint.}
Indeed, $D$ is a direct sum and if $F=(f,\phi) \in\Dom(D^\ast)$ then 
$\phi\in\Dom(d^\ast)$ and $f\in\Dom(\delta^\ast)$ and then, by hypothesis, 
$\phi\in\Dom(\delta_{min})$ and $f\in\Dom(d_{min}),$ thus $F\in \Dom(\bar D).$

For the following we need some formulas taken in \cite{M}. First we set, for each $f \in C^0(K)$
\begin{equation}\label{moyenne}
  \tilde f(e)=\frac{1}{2}(f(e^+)+f(e^-)).
\end{equation}
The function $\tilde f$ is  even  on the edges. We have
\begin{align}\label{diprod}\forall f,g \in C^0(K),\, 
\forall e\in\mcE, \quad
\di(fg)(e)&=f(e^+)\di g(e)+\di f(e) g(e^-)\\
&=\tilde f(e)\di g(e)+\tilde g(e)\di f(e)\nonumber\\
\label{delprod}\forall f\in C^0(K),\phi\in C^1(K),\,
\forall v\in\mcV, \quad \delta(\tilde f \phi)(v)&=f(v)\delta\phi(v)-
\frac{1}{2c(v)}\sum_{e^+=v}r(e)\di f(e)\phi(e).
\end{align}
We prove now these two equalities.
\subsubsection*{(b) If $f\in\Dom(\di_{max})$ then $\|(f-\chi_nf)\|+\|\di(f-\chi_nf)\|\to 0$ when $n\to\infty$}
This will show that $\di_{min}=\di_{max}.$ 

Let $f\in\Dom(\di_{max}),$ we can then calculate
\begin{gather*}\|(f-\chi_nf)\|^2\leq\sum_{v\notin B_n}c(v)|f(v)|^2
\stackrel{n\to\infty}{\longrightarrow} 0
\end{gather*}
because $f\in L_2(\mcV)$.
For the second term, the relation \eref{diprod} gives
$$\di(f-\chi_nf)(e)=(1-\chi_n)(e^+)\di(f)(e)-f(e^+)\di(\chi_n)(e).
$$
Because of \eref{E2} (and with an abuse of notation),
\begin{gather*}
\lim_{n\to\infty}\|(1-\chi_n)(e^+)\di(f)(e)\|=0
\end{gather*}
On the other hand, 
\begin{align*}\|f(e^+)\di(\chi_n)(e)\|^2&=\sum_{e\in\mcE}r(e)|f(e^+)|^2
|\di(\chi_n)(e)|^2\\
&=\sum_{x\in\mcV}|f(x)|^2\sum_{e^+=x}r(e)|\di(\chi_n)(e)|^2\\
&\leq\sum_{x\in\mcV,\exists e\in\supp(d\chi_n),e^+=x}C c(x)|f(x)|^2
\end{align*}
by the hypothesis (iii).
The property \eref{E} permits to conclude that this term tends to 0
as $n\to\infty.$
\subsubsection*{(c) If $\phi\in\Dom(\delta_{max})$ then 
$\|(\phi-\tilde\chi_n\phi)\|+\|\delta(\phi-\tilde\chi_n\phi)\|\to 0$ when $n\to\infty$}
This will show that $\delta_{min}=\delta_{max}.$

Let $\phi\in\Dom(\delta_{max}),$ by the properties \eref{E0} and \eref{E} 
we know that
$$\forall p\in\N, \;\forall n\geq n_p, \quad \|\phi-\tilde\chi_n\phi\|^2\leq
\sum_{e\in\mcE_p^c}r(e)|\phi(e)|^2
$$
so 
$\displaystyle\lim_{n\to\infty}\|\phi-\tilde\chi_n\phi\|=0.$

On the other hand, by \eref{delprod}
\begin{align*}
\delta(\phi-\tilde\chi_n\phi)(v)
&=\delta\left(\wt{(1-\chi_n)}\phi\right)(v) \\
&= (1-\chi_n)(v)\delta\phi(v)+\frac{1}{2c(v)}\sum_{e^+=v}r(e)\di\chi_n (e)\phi(e)
\end{align*}
Clearly
$$\lim_{n\to\infty}\|(1-\chi_n)\delta\phi\|=0
$$
because $\delta\phi\in L_2(\mcV).$ For the second term, we use (iii)
and the Cauchy-Schwarz inequality:
\begin{align*} 
\forall v\in\mcV,\, \Big \vert\sum_{e^+=v}r(e)\di\chi_n (e)\phi(e)\Big \vert^2
&\leq\sum_{e^+=v}r(e)|\di\chi_n (e)|^2\sum_{e\in\supp(\di\chi_n),e^+=v}r(e)|\phi(e)|^2
\\
&\leq Cc(v)\sum_{e\in\supp(\di\chi_n),e^+=v}r(e)|\phi(e)|^2 \\
\hbox{so, }\sum_{v\in\mcV}c(v)\Big\vert\frac{1}{2c(v)}\sum_{e^+=v}r(e)\di\chi_n (e)
\phi(e)\Big\vert^2
&\leq C\sum_{v\in\mcV}\sum_{e\in\supp(\di\chi_n),e^+=v}r(e)|\phi(e)|^2\\
&\leq C \sum_{e\in\supp(\di\chi_n)}r(e)|\phi(e)|^2.
\end{align*}
This term tends to $0$ by properties \eref{E0} and \eref{E}.
\end{proof}
\begin{pro}\label{esa2}
Let K be a connected, locally finite graph. The operator $D$  is  essentially self-adjoint
if and only if the operator $\Delta$ is  essentially self-adjoint.
\end{pro}
\begin{proof}\tir If $D$ is essentially self-adjoint, then $\Im(D\pm i)$ are dense and $(\bar D\pm i)$ are invertible.
This is a result for essentially self-adjoint operators (Corollary of Theorem VIII.3 in \cite{RS}).
By the second property we know that  
\begin{equation}\label{resolv}
\exists C_2>0, \forall F\in\Dom(\bar D),\, \|F\|_{L_2}\leq C_2 \|(\bar D\pm i)(F)\|_{L_2}.
\end{equation}
Note also that 
$$D(C_0^0(K)\oplus C_0^1(K))\subset C_0^0(K)\oplus C_0^1(K).$$

Now, by the theorem of von Neumann, $(\bar D)^2=D^\ast \bar D$ is self-adjoint when $D^\ast=\bar D$ and it is an 
extension of $\Delta.$ As a consequence, the domain of $(\bar D)^2$ contains the domain of $\bar\Delta,$ the 
closure of $\Delta$. But 
$$\Dom(\bar\Delta)\subset\Dom((\bar D)^2)\Rightarrow\Dom((\bar D)^2)\subset \Dom(\Delta^\ast).
$$
In fact, we have also $\Dom(\Delta^\ast)\subset \Dom((\bar D)^2)$: let $\Psi\in\Dom(\Delta^\ast),$ then
$$\exists C_1>0, \,\forall F\in  C_0^0(K)\oplus C_0^1(K),~~ |\langle (\Delta+1)(F),\Psi\rangle |\leq C_1 \|F\|_{L_2}.
$$
We now consider  the linear form defined on $C_0^0(K)\oplus C_0^1(K)$, by 
$$G\longmapsto \left\langle(D-i)G,\Psi \right\rangle
$$ 

For all $G\in\Im(D+ i),$ $\exists F\in  C_0^0(K)\oplus C_0^1(K),$
such that $G=(D+i)(F)$ so $G\in C_0^0(K)\oplus C_0^1(K)$ and, using \eref{resolv}
$$|\langle (D-i)G,\Psi\rangle |=|\langle (\Delta+1)F,\Psi\rangle |\leq C_1 \|F\|_{L_2}\leq C_1 C_2\|G\|_{L_2}$$

Hence 
\begin{equation}
\exists C>0,\, \forall G\in \Im(D+ i),\, |\langle (D-i)G,\Psi\rangle |\leq C \|G\|_{L_2}.
\end{equation}

But $\Im(D+ i)$ is dense, it means that the considered linear form extends continuously on $L_2$ 
or that $(D+i)\Psi\in L_2.$
Thus $\Psi\in\Dom( \bar D)$ because $\bar D$ is self-adjoint. It is then clear that 
$D(\Psi)\in\Dom( \bar D)$:
$$\forall F\in  C_0^0(K)\oplus C_0^1(K),~~ |\langle D(F),D(\Psi)\rangle |=|\langle \Delta(F),\Psi\rangle |
\leq (C_1+\|\Psi\|_{L_2})\|F\|_{L_2}.
$$
So, we have proved
$$\Dom(\Delta^\ast)\subset\Dom((\bar D)^2)\Rightarrow\Dom((\bar D)^2)\subset\Dom(\bar\Delta)
$$
 because $\Delta^{\ast\ast}=\bar\Delta,$ and finally
$$\Dom(\bar\Delta)=\Dom((\bar D)^2)
$$
and then $\bar\Delta=(\bar D)^2$ is self-adjoint.

Let us now look at the converse: if $\Delta$ is essentially 
self-adjoint then, by the Corollary of Theorem VIII.3 in \cite{RS}, $\Im(\Delta+1)$ is dense but
$$\Delta+1=(D+i)(D-i)=(D-i)(D+i)\Rightarrow\Im(\Delta+1)\subset\Im(D\pm i).
$$
Thus $\Im(D\pm i)$ are both dense and $D$ is essentially self-adjoint.
\end{proof}
This proof essentially follows \cite{Ch}, it uses in a very significant way the fact that $D$ maps 
elements of finite support into themselves.
\begin{cor}\label{esa2}
Let K be a connected, locally finite graph. If K is $\chi$-complete, 
then the operator $\Delta$ is  essentially self-adjoint.
\end{cor}

\begin{rem} \tir
The case studied in  \cite{T}, namely a complete graph for the metric 
 $$a(e)=\frac{\sqrt{c(e^+)c(e^-)}}{r(e)}
 $$ 
and with a valence bounded by $ A $ can be handled with the same kind of calculus,
although it is not clear that this metric is intrinsic. Indeed, with the same $\chi_n$
(defined by \eref{chin}), the bound now satisfied is
$$\exists C>0\; , ~\forall e\in\mcE ,~n\in\N ~,\quad  r(e)\di\chi_n(e)^2\leq C
\sqrt{c(e^+)c(e^-)}.
$$ 
We write
\begin{multline*}
\sum_{e\in\mcE}r(e)f(e^-)\di\chi_n(e)\bar\phi(e)
=\frac 12\sum_{e\in\mcE}r(e)(f(e^+)+f(e^-))\di\chi_n(e)\bar\phi(e)\\
\leq\frac 12\sqrt{\sum_{e\in\supp(\di\chi_n)}r(e)|\phi(e)|^2}
\sqrt{\sum_{e\in\supp(\di\chi_n)}r(e)|f(e^+)+f(e^-)|^2\di\chi_n(e)^2}
\end{multline*}
\begin{align*}\hbox{and } & \sum_{e\in\supp(\di\chi_n)}r(e)|f(e^+)+f(e^-)|^2\di\chi_n(e)^2\\
&=\sum_{e\in\supp(\di\chi_n)}r(e) \Big [|(f(e^+)-f(e^-)|^2+
4\Re \Big(f(e^+)\bar f(e^-)\Big) \Big]\di\chi_n(e)^2\\
&=\sum_{e\in\supp(\di\chi_n)}r(e)|\di(f)(e)|^2+4
\Re \Big(\sum_{x\in\mcV}f(x)\sum_{e^+=x}r(e) \bar f(e^-)\di\chi_n(e)^2\Big)
\end{align*}
the first term tends to $0$ by completeness and the second is bounded as follows
\begin{multline*}
\Re \Big(\sum_{x\in\mcV}|f(x)|\sum_{e^+=x}r(e)|f(e^-)|\di\chi_n(e)^2\Big)\\
\leq C\sum_{x\in\mcV}|f(x)|\sum_{e\in\supp\di\chi_n,e^+=x}|f(e^-)|\sqrt{c(e^+)c(e^-)}\\
\leq AC\sum_{x\in\mcV,\exists e\in\supp\di\chi_n,e^+=x}c(x)|f(x)|^2
\end{multline*}
because, as $\mcE$ is symmetric, one has
$$\sum_{x\in\mcV,\exists e\in\supp\di\chi_n,e^+=x}c(x)|f(x)|^2=
\sum_{x\in\mcV,\exists e\in\supp\di\chi_n,e^-=x}c(x)|f(x)|^2
$$
So the second term also tends to 0, because of completeness and bounded valence.
\end{rem}

\section{Flanders Theorem}
\subsection{Flanders problem}
In 1971, Flanders published a very nice result \cite{F} concerning resistive 
networks.
The problem is the following (keeping the notations of Flanders): Let $i$ be a finite current 
source, {\it i.e.} an element of $C_0^0(K)$, 
and $E'$ a finite voltage source,  {\it i.e.} an element of $C_0^1(K)$, 
\begin{center}
is there a resulting current flow, and is it unique? 
\end{center}
 {\it i.e.} find $L_2$-solutions $I$ of the problem (Kirchhoff's laws):\begin{equation}\label{kirchhoff}
  \left\{\begin{array}{ll}
\hbox{ (Kirchhoff's current law)} & \delta(I)+i =0, 
\cr  \hbox{ (Kirchhoff's voltage law)} & \forall Z,\, \partial Z=0,
\quad \displaystyle \int_Z E'=\int_Z I ,
  
\end{array}
\right.
\end{equation}

Here $Z$ is a cycle,  {\it i.e.}  a 1-chain (a formal finite sum of oriented edges)  with no boundary. 

Geometrically, if we  write $ Z=\sum_{e\in\mcE^+} z_e e,\,z_e\in\Z,$ the boundary 
$\partial$ of a 1-chain is an operator defined on the edges by $\partial(e)=e^+-e^-.$ 

On the sequel we will prefer a skew-symmetric notation: 
\begin{align}
Z&=\frac 12\sum_{e\in\mcE} z_e e,\,z_e\in\Z\hbox{, with } z_e=-z_{-e} \nonumber\\
\label{bord}\partial Z&=\sum_{x\in\mcV}(\sum_{e^+=x} z_e)x.
\end{align}
The integral in \eref{kirchhoff} has to be understood in the simplicial 
framework:
\begin{equation}\int_Z I=\frac 12\sum_{e\in\mcE}z_e I(e)
\end{equation}

Flanders studies this problem for an infinite graph with weight $c=1$ on
vertices (Remark that our weight on edges is in fact the inverse of the 
resistances $r$ introduced by Flanders, so our unknown $I$ corresponds to 
$r.I$ in the notations of Flanders).
He shows that this problem has a unique $L_2$-solution which is the limit of 
finite flows ({\it ie.} solutions on an increasing sequence of finite 
subgraphs) if $i$ has zero mean value $\displaystyle \sum_{v\in\mcV} i(v)=0.$

\subsection{Flanders type Theorem} \label{flanders}
In the framework we have introduced in Section 2,  this question is related 
to the question of the Hodge decomposition.
Indeed, the second condition tells us that the periods of $I$ are given by 
those of $E'$, this determine the harmonic component  of $I$, {\it i.e.} the 
orthogonal projection of $I$ on $\Ker(\delta)$, while the complementary component
must be sent by $\delta$ on $-i$. 
So we have to look for $I=E_0+I_0$ such that 
$E_0$ is the harmonic component of $E'$ and $I_0$ satisfies  $-i=\delta(I_0)$ 
and $\int_Z I_0=0$ on cycles.
\begin{rem}\label{minim} Indeed, this question is related with the {\em uniqueness problem} studied
very carefully in a lot of works, we refer to \cite{LP} for a precise presentation. 
It appears that, at least with finite source current and no voltage current, the situation mostly studied,
({\it i.e.} $i$ has finite support and $E'=0$), the two general solutions are the 
{\em free current} which is the solution proposed by Flanders, and the {\em wired current} which is 
the solution of {\em minimal energy} ({\it i.e.} of minimal $L_2$-norm). It is clear, with the previous
decomposition $I=E_0+I_0,$ that there exists at most one $I_0$ and the solution of
minimal energy is given by the choice of $E_0$ with minimal norm. When $E'=0$ this solution is $E_0=0.$ 
The {\em uniqueness problem} is to find conditions where there is no choice although for $E_0.$
We will not study this question here but focus on the existence question which concerns in fact $I_0.$  
\end{rem}
\begin{dfn} Any cycle $Z$ defines a unique 1-cochain $E_Z\in C_0^1(K)$ such that $\forall E\in L_2(\mcE)$
$$\int_Z E=\langle E,E_Z\rangle$$
and we have the formula
$$Z=\sum_{e\in\mcE^+} z_e e,\,z_e\in \Z\quad\Rightarrow 
E_Z=\sum_{e\in\mcE^+} \frac{z_e}{r_e}e^*.
$$
where the cochain $e^*$ is defined by $e^*(e)=1$ and $e^*(e')=0$ if 
$e'\not=\pm e.$ 

An $L_2$-cycle $Z$ is an (infinite) cycle such that $E_Z\in L_2(\mcE).$
\end{dfn}
\begin{lem} 
For any $L_2$-cycle $Z$ the 1-cochain $E_Z$ satisfies formally that $\delta E_Z=0.$ Hence
$$E_Z\in\Ker\delta_{\rm max}.
$$
\end{lem}
This result is a simple consequence of \eref{bord}.
\begin{lem}\label{period}
 For any $E\in L_2(\mcE)$ orthogonal to $\Ker\delta_{\rm max}$ and any $L_2$-cycle $Z$
$$\int_Z E=0.
$$
\end{lem}
\begin{proof} \tir
Indeed, for any $L_2-$cycle $Z$, $E_Z\in \Ker\delta_{\rm max}\subset L_2(\mcE)$ and 
$$\int_Z E=\langle E,E_Z\rangle=0.$$
\end{proof}
\begin{rem} The uniqueness problem is then related to sufficient conditions for $\Ker\delta_{\rm max}$ to be
generated by the $E_Z,$ a priori we could consider finite cycles, or $L_2$-cycles. If 
$\phi\in\Ker\delta_{\rm max}\subset L_2(\mcE)$
is orthogonal to any $E_Z,Z$ finite cycle, then there exists a function $f\in C^0(K)$ such that
$\phi=\di f$, so there exists a harmonic function with finite energy, but not necessarily $L_2.$
\end{rem}
\begin{thm}
Let $K$ be a connected, locally finite graph. We suppose that it is $\chi-$complete such that the 
operator $D$ defined on $C_0^0(K)\oplus C_0^1(K)$ is essentially self-adjoint. Then for any 
$i\in C_0^0(K)$ satisfying $\displaystyle \sum_{v\in\mcV}c(v) i(v)=0$ and for any $E'\in L_2(\mcE)$ there 
exists a unique solution of minimal energy $I\in \Dom(\bar\delta)$ of the problem:
\begin{equation}\label{kirchhoffL2}
\delta(I)+i=0, \hbox{ and $\forall Z,\, L_2-$cycle}\quad\int_Z E'=\int_Z I.
\end{equation}
\end{thm}
\begin{proof} \tir
By hypothesis, $\bar\delta=\delta_{\rm max}.$ The space $\Ker\bar\delta$ is closed in $L_2(\mcE)$, so any 
element $I\in L_2(\mcE)$ can be writen $I=E_0+I_0$ with $E_0\in\Ker\bar\delta$ and $I_0$ in its orthogonal 
complement. By \lref{period}, 
and Remark \ref{minim}, if we look for a solution of minimal energy, $E_0$ must be the orthogonal 
projection of $E'$ on $\mcZ\subset\Ker\bar\delta,$ the closure of the vector space generated by all 
the $E_Z, Z\; L_2\hbox{-cycle }$ and then, by definition of the $E_Z,$
$$\forall Z \,  L_2\hbox{-cycle } , ~ \int_Z E'=\int_Z E_0.
$$

Now, the existence of $I_0$ is related to the property of $-i$ to be in the range of $\bar\Delta.$
In the case where $i$ has finite support, we can do as follows: let $K_0$ be a finite connected 
subgraph of $K$ (see \ref{subgraph}, vertices of $K_0$ are vertices of $K$ and edges of $K_0$  are 
edges of $K$). We suppose that the support of $i$ is included in $K_0.$ 
Denote by $d_0$ the difference
operator of $K_0.$ The Laplacian $\Delta_0$ of $K_0$ is
self-adjoint and $\Im \Delta_0=\Ker \Delta_0^\perp.$ Thus, as $\Ker \Delta_0=\R$ 
consists of constant functions
$$\langle i,1\rangle =0\Rightarrow \exists f\in C^0(K_0), \; -i=\Delta_0(f).
$$
Let $\phi\in C^1_0(K)$ be  the extension of $\di _0 f$ by $0$ on the edges 
which don't belong to $K_0$.
This form is certainly different from $\di f$ but $\delta\phi=-i.$

We define now $I_0$ as the orthogonal projection of $\phi$ on the orthogonal complement
of $\Ker\bar\delta,$
it means that $I_0$ differs from $\phi$ by an element of $\Ker\bar\delta$ and that 
$I_0\in\Ker\bar\delta^\perp.$
Using \lref{period}, we conclude that:
$$ \delta I_0=-i \hbox{ and }\forall Z\, L_2\hbox{-cycle }, ~ \int_Z I_0=0.
$$
Thus $I_0+E_0$ is the solution of the problem with minimal energy.
\end{proof}
\begin{rem}\tir In the original theorem of Flanders, $E'$ has finite support, and we only 
take care of finite cycles, but the proof extends easily to $E'\in L_2(\mcE)$ if we 
consider only $L_2$-cycles. The question is how to extend it to more general $i$ or can we
characterize $\Im(\bar\delta)?$ If we can prove that $\Im(\bar\delta)$ is closed, then the 
Hodge decomposition applies (see \eref{hodge}) and the answer will be quite simple;  that 
is what we explore below.
\end{rem}
\subsection{Anghel's hypothesis}\label{anghel}
In \cite{A}, N. Anghel shows that a Dirac type operator $D$ defined on a 
complete manifold is Fredholm if and only if $D^2$ is {\it positive at 
infinity.}
 
\vspace{0.6cm}
Let us define the complementary of a subgraph of a graph.
\begin{dfn}
 For  a subgraph  $K_0$ of a graph $K$, we define
the \emph{complementary graph} $K_0^c=(\mcV^c,\mcE^c)$ as follows
$$\mcV^c=\mcV\smallsetminus\mcV_0,\quad \mcE^c=\{e\in\mcE\smallsetminus \mcE_0,~\partial(e)\subset\mcV^c\}.
$$
\end{dfn}
\begin{rem}
\begin{enumerate}
\item In particular boundary points of edges in $\mcE_0$ belong to $\mcV_0.$
\item As a consequence of the definition, $\mcE^c$ avoids the edges with one end in
$\mcV^c$ and one in $\mcV_0.$
\end{enumerate}
\end{rem}
Following \cite{KL}, we define the {\it boundary} of a subgraph $K_0$ to be 
its edge boundary~: 
$$\partial(K_0)=\mcE\smallsetminus(\mcE_0\cup\mcE^c).
$$

\begin{dfn}
 We say that a closed Dirac type operator $D$ is \emph{positive at infinity} if
there exists a finite connected subgraph $K_0=(\mcV_0,\mcE_0)$ of $K$ such that
\begin{equation}\label{posinf}
\exists C>0, \quad\forall (f,\phi)\in L_2(\mcV^c)\times L_2(\mcE^c) \cap\Dom(D), 
\quad\|(f,\phi)\|\leq C\|D(f,\phi)\|
\end{equation}
where $D(f,\phi)$ is in fact $D$ applied to the prolongations by 0 of $(f,\phi).$
\end{dfn}
(Remark that this definition gives rather positivity of $\Delta$.)
\begin{thm}
If the graph (connected and locally finite) is $\chi-$complete and 
if its Gau\ss-Bonnet operator 
$$D=\di+\delta
$$ 
(which is essentially self-adjoint) satisfies that $\bar D$ is positive at infinity, then $\Im(\bar D)$ is 
closed and, as a consequence,  the Hodge property holds~:
\begin{equation}\label{hodge}
L_2(\mcE)=\Ker \bar\delta\oplus\Im(\bar\di),\quad 
L_2(\mcV)=\Ker \bar\di\oplus\Im(\bar\delta).
\end{equation}
\end{thm}
\begin{proof}\tir
The condition \eref{posinf} implies that the closed
restriction operator $D^c$ of $\bar D$ on $K_0^c:$
$$D^c:\Dom(D^c)\subset L_2(\mcV^c)\times L_2(\mcE^c) \to L_2(\mcV^c)\times 
L_2(\mcE^c)
$$
is continuous (for the graph norm on $\Dom(D^c)$), injective and with closed
image. By the inversion theorem, there exists
$$P:L_2(\mcV^c)\times L_2(\mcE^c)\to \Dom(D^c)
$$
such that $P\circ D^c=\I ,\quad \hbox{and}\quad \I-D^c\circ P$ is the orthogonal projector
on the subspace $\Im(D^c)^\perp.$

Let now $\psi\in\overline{\Im(\bar D)}$. It means:
$$\exists \hbox{ a sequence }(\sigma_n)_{n\in \N} ~\hbox{in} ~\Dom(\bar D),\quad 
\sigma_n\in\Ker(\bar D)^\perp,\hbox{ and }\lim_{n\to\infty}\bar D(\sigma_n)=\psi.
$$
{\it The sequence $(\sigma_n)$ is bounded.} If not, $(\sigma_n)$ admits a subsequence whose norm tends 
to $+\infty,$ denoting this subsequence $(\sigma_n)$ again, we construct
$$\phi_n=\dfrac{\sigma_n}{\|\sigma_n\|}.
$$ 
It satisfies
$$\|\phi_n\|=1,~ \lim_{n\to\infty}\bar D(\phi_n)=0.
$$
Then the restriction of $\bar D(\phi_n)$ to $K_0^c$ also 
converge to $0$ in $L_2(\mcV^c)\times L_2(\mcE^c)$.\\
 But the set of vertices not in $\mcV^c$ and the set of edges 
not in $\mcE^c$ are finite. As $\phi_n$ is bounded, by passing to a subsequence, we can 
suppose that all their values in these finite sets converge, and by the 
same argument we can suppose that the value of $\phi_n$ on the vertices 
which are boundary points of edges in $\partial(K_0)$ converge. By local 
finiteness we conclude that $\bar D({\phi_n}_{|K_0^c})$ converges.

By \eref{posinf}, then also ${\phi_n}_{|K_0^c}$ converges, thus finally 
$\phi_n$ converges, let $\phi$ be the limit, it satisfies
$$\|\phi\|=1, ~\phi\in\Ker(\bar D)^\perp , ~\bar D(\phi)=0.
$$
Thre is a contradiction.

So we can suppose that $(\sigma_n)$ is bounded, then by the same kind of 
reasoning, we show that $(\sigma_n)_n$ admits a subsequence which converges, 
let $\sigma$ be this limit. As $\bar D$ is closed and $\bar D(\sigma_n)$
converges, then $\sigma\in\Dom(\bar D)$ and $\bar D(\sigma)=\psi.$ 
\end{proof}
We see that the reasoning is separated for $0-$forms and $1-$forms. This gives:
\begin{cor}
Let $K$ be a graph (connected and locally finite)  $\chi-$complete
so its Gau\ss-Bonnet operator $D=\di+\delta$ is essentially 
self-adjoint. If $\di$ satisfies the condition
\begin{equation}\label{posinfd}
\exists C>0, \quad\forall f\in L_2(\mcV^c) \cap\Dom(\bar \di), 
\quad\|f\|\leq C\|\bar \di f\|
\end{equation}
for the complementary of some finite graph, then $\Im \bar\di$ is closed and
$$L_2(\mcE)=\Ker \bar\delta\oplus\Im(\bar\di).
$$
And there exists a similar statement for $\delta.$
\end{cor}

\section{Examples} \label{examples}
It is clear that if $K$ possesses infinitely many cycles (as infinite ladders, 
or infinite grids), the condition of positivity at infinity will not be satisfied
because there will be elements in $\Ker\delta$ with support as far as we want.
A family of examples could be a graph with finite geometry: there exists a
finite subgraph $K_0$ such that $K_0^c$ is a disconnected (finite) union of  
branches.
\begin{figure}[ht]
\includegraphics*[height=6cm,width=8cm]{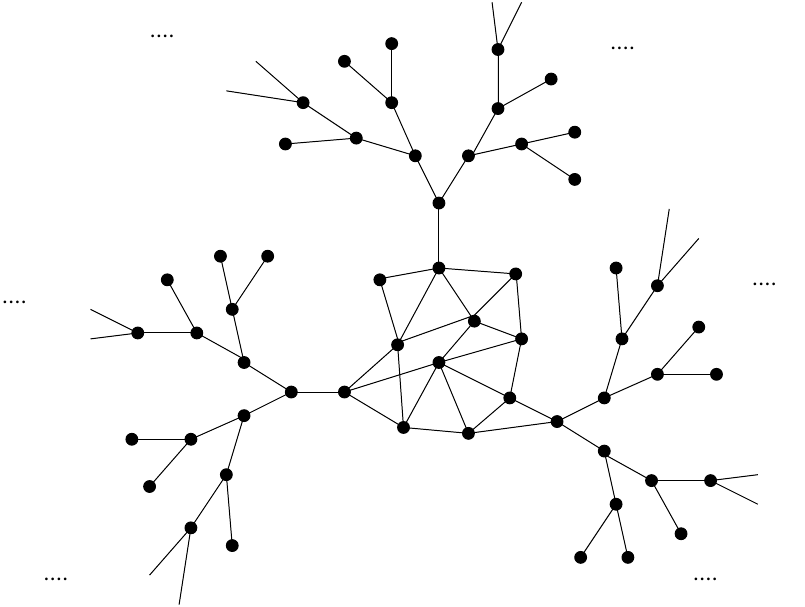}
\end{figure}

\begin{pro}\label{arbre}
If the connected graph  $K$ admits a finite subgraph 
such that its complementary is a finite union of trees with constant valence
larger than 3, 
then, considered with the weights constant equal to $1$ on vertex and edges, 
it is $\chi-$complete and $\Im\bar\di$ is closed.
\end{pro}

\begin{proof} \tir
We will prove that $\di$ is positive at infinity, {\it ie.} on each tree.
Let $U$ be a tree with a base point and valence $p+1,\,p\geq 2.$
We apply Corollary 17 of \cite{KL}, taking the notations of this paper
(in particular $\sharp $ denotes the cardinality):
in our case $D_U=p+1$ is finite, so it 
suffices to show that the isoperimetric constant $\alpha_U$ is positive.\\
Recall that 
\begin{equation}\label{constantiso}
\alpha_U=\inf_{W\subset U,\hbox{finite}}\frac{\sharp(\partial W)}{\sharp W}.
\end{equation}

For a tree, one has a notion of {\it height}: the base point is of height $0,$
and for another point its height is the necessary number of edges  to 
join it to the base point.

Let $W$ be a finite set of vertices of $U$, we shall show by reccurence
on $\sharp W$ that 
$${\sharp(\partial W)}\geq{\sharp W}.
$$
If $\sharp W=1,$ then $\sharp(\partial W)=p+1.$
If $\sharp W=n\geq 1,$ let $x\in W$ be a point of highest height in $W$ and
$y$ is the point just below.
Then define $W'=W-\{x\}$ so $\sharp W'=\sharp W-1$ and 
\begin{align*}
y\in W\Rightarrow\sharp(\partial W)=p-1+\sharp(\partial W')\\
y\notin W\Rightarrow\sharp(\partial W)=p+1+\sharp(\partial W')
\end{align*}
In all cases, applying the reccurence hypothesis, we get:
$$\sharp(\partial W)\geq p-1+\sharp(\partial W')\geq p-1+\sharp W-1\geq\sharp W.
$$
\end{proof}
\begin{cor}
Such a graph (as in the proposition \ref{arbre}) satisfies
also that $\Im\bar\delta$ is closed and $\Ker\bar\di=\{0\}$ (because constants 
are not in $L_2$), so $\bar\delta$ is surjective. \\
As a consequence, for such a graph Flanders problem \eref{kirchhoff} has always a unique solution 
with minimal energy.
\end{cor}
\begin{proof} \tir
Indeed, if \eref{posinfd} is satisfied, then 
\begin{equation}\label{posinfDelta}
\forall f\in\Dom(\Delta^c)\subset L_2(\mcV^c),\quad \|f\|\leq C^2\|\Delta(f)\|.
\end{equation}
Thus, by the same reasoning as before the range of $\bar\Delta$ acting on functions
is closed. Now if $(\phi_n)_n$ is a sequence of 1-forms such that $\delta(\phi_n)$
converges, we can apply the Hodge decomposition \eref{hodge} at $\phi_n$, because of the 
  \pref{arbre}:
$$\exists f_n\in\Dom (\bar\di)\hbox{ such that }\delta\circ\di(f_n)\in L_2(\mcV) 
\hbox{ and converges.}
$$
But we can extract a subsequence of $(f_n)_n$ which converges, because of 
\eref{posinfDelta}.
\end{proof}
\begin{pro}\label{arbre2}
If the connected graph  $K$ admits a finite subgraph such that its complementary is a finite union of trees with  valencelarger than 3, then, considered with the weights equal to the valence on vertices and constant equal to $1$ on  edges, it is $\chi-$complete  and $\Im\bar\di$ is closed.
\end{pro}

\begin{proof} \tir It is clear that such a graph satisfies the condition of 
$\chi-$completeness. 
The fact that $\di$ is positive at infinity is again a  consequence of the 
results of \cite{KL}. Indeed, by hypothesis we have 
$\forall v\in\mcV, m(v)=\sharp\{e\in\mcE,e^+=v\}$ at least on the "tree-part", 
thus is it equal to the function $n$ introduced in \cite{KL} and their $d$ 
is constant equal to 1. By their 
Proposition 15, the quadratic form on a part $U$ is bounded from below by 
$1-\sqrt{1-\alpha_U^2}$ if 
$\alpha_U$ is the isoperimetric constant introduced in (\ref{constantiso}) 
 but now with the volumes $|.|$ defined by the weights: 
$$\alpha_U=\inf_{W\subset U,\hbox{finite}}\frac{|\partial W|}{|W|}. $$

Let $W$ be a finite part of a tree. Its number of (oriented) edges is 
$\displaystyle \sum_{v\in W}m(v)=|W|.$ 
But, because it is in a tree the number of interior edges is at most 
$2.\sharp (W)$. Thus 
$$\frac{|\partial W|}{|W|}\geq \frac{\sum_{v\in W}(m(v)-2)}{\sum_{v\in W}m(v)}
\geq \frac{1}{3}
$$
because $m(v)\geq 3.$\end{proof}
The same Corollary as before holds, for the same reasons.
\begin{cor} 
Such a graph (as in the Proposition \ref{arbre2}) satisfies
also that $\Im\bar\delta$ is closed and $\Ker\bar\di=\{0\}$ (because constants 
are not in $L_2$), so $\bar\delta$ is surjective. \\
As a consequence, for such a graph,  Flanders problem \eref{kirchhoff} has 
always a unique solution with minimal energy.
\end{cor}
\begin{rem}
Take care to the fact that in these situations $\Ker\bar\delta$
can be non trivial~:  on a tree of valence 3, with all the weights equal to 1,
fix a point $O,$ it has at least two edges which go to infinity: $(x,O)$ and
$(y,O).$ \\

\begin{figure}[ht]
\includegraphics*[height=6cm,width=9cm]{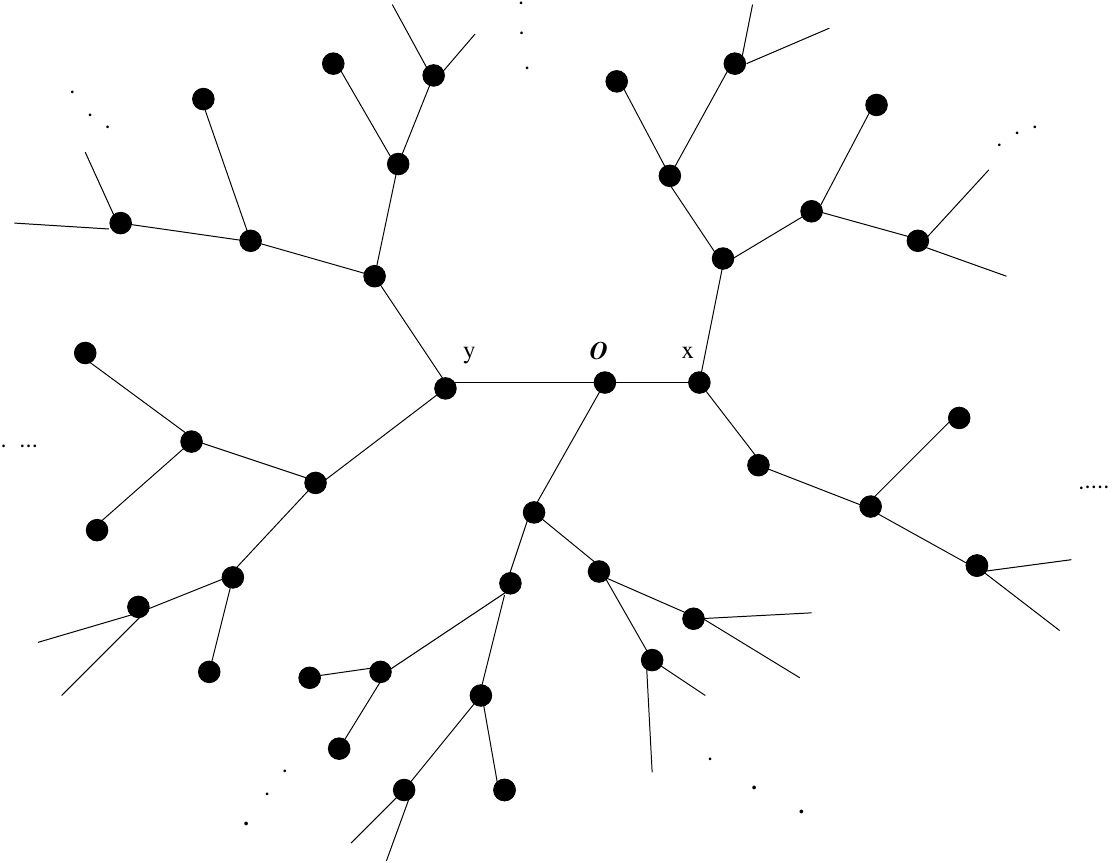}
\end{figure}

 Let $\phi$ be the form such that 
$$\phi(x,O)=1,\;\phi(y,O)=-1
$$
at the $n$-level on the branch emanating from $x$ we put the value
of $\phi$ to be $\dfrac{1}{2^n},$ and at the $n$-level on the branch 
emanating from $y$ we put the value of $\phi$ to be $\dfrac{-1}{2^n}.$
Elsewhere, we put $\phi(e)=0.$\\
It is easy to verify that such a $\phi$ is in $L_2$ and satisfies 
$\delta(\phi)=0,$ see also \cite{Ay}.
\end{rem}
\begin{rem}In these two last cases the Laplacian is bounded, and the non zero spectrum is bounded from  below because the isoperimetric constant $\alpha_U$ admits a bound independent on $U$.
\end{rem}

\vspace{1cm}
{ \bf \textit{ Acknowledgements}} Part of this work was done while the author N.T-H  was visiting the University of Nantes. She would like to thank the Laboratoire de Math\'ematiques  Jean Leray  (LMJL) for its hospitality. She is greatly indebted to the research unity (UR / 13 Z S 47)  for its continuous support.\\
This work was supported by Grants through both  G\'eanpyl project (FR  2962 du CNRS Math\'ematiques des Pays de Loire) and PHC-Utique (13 G 15-01) "Graphes, g\'eom\'etrie et th\'eorie spectrale". 
\\ 
The authors thank Sylvain Golenia, Matthias Keller and Ognjen Milatovic for 
their reading with great interest and for their remarks. They would like to thank also 
the anonymous referee for their numerous relevant remarks and useful suggestions.


\end{document}